\documentclass[11pt]{amsproc}
\usepackage{amsmath}
\usepackage{amssymb}
\usepackage{amsfonts}

\newtheorem{theorem}{Theorem}
\newtheorem{proposition}{Proposition}
\newtheorem{corollary}{Corollary}
\newtheorem{example}{Example}
\newtheorem{remark}{Remark}

\newtheorem{lemma}{Lemma}
\newtheorem{definition}{Definition}
\newtheorem{notation}{Notation}

\input{tcilatex}

\begin{document}
\title{Unipotent group actions on affine varieties}
\author{H.Derksen$^{\ast }$}
\address{Harm Derksen\\
Department of Mathematics\\
University of Michigan, \\
Ann Arbor, MI 48109 \\
USA\\
hderksen@umich.edu}
\thanks{$^{\ast }$Funded by NSF CAREER grant, DMS 0349019, Invariant Theory,
Algorithms and Applications. \ }
\author{A.van den Essen}
\address{Arno van den Essen \\
Radboud University Nijmegen\\
Toernooiveld 1, The Netherlands\\
a.vandenessen@math.ru.nl \ }
\author{D. R.Finston}
\address{David R. Finston\\
Department of Mathematical Sciences\\
New Mexico State University, \\
Las Cruces, NM 88003 USA\\
dfinston@nmsu.edu \ }
\author{S.Maubach$^{\ast \ast }$ }
\address{Stefan Maubach \\
Radboud University Nijmegen\\
Toernooiveld 1, The Netherlands\\
s.maubach@math.ru.nl}
\thanks{$^{\ast \ast }$Funded by Veni-grant of council for the physical
sciences, Netherlands Organisation for scientific research (NWO)}
\maketitle

\begin{abstract}
Algebraic actions of unipotent groups $U$ actions on affine $k-$varieties $X$
($k$ an algebraically closed field of characteristic 0) for which the
algebraic quotient $X//U$ has small dimension are considered$.$ In case $X$
is factorial, $O(X)^{\ast }=k^{\ast },$ and $X//U$ is one-dimensional, it is
shown that $O(X)^{U}$=$k[f]$, and if some point in $X$ has trivial isotropy,
then $X$ is $U$ equivariantly isomorphic to $U\times A^{1}(k).$ \ The main
results are given distinct geometric and algebraic proofs. Links to the
Abhyankar-Sathaye conjecture and a new equivalent formulation of the Sathaye
conjecture are made.
\end{abstract}

\section{Preliminaries and Introduction}

Throughout, $k$ will denote a field of characteristic zero, $k^{[n]}$ the
polynomial ring in $n$ variables over $k,$ and $U$ a unipotent algebraic
group over $k.$ Our interest is in algebraic actions of such $U$ on affine $%
k-$varieties $X$ (equivalently on their coordinate rings $\mathcal{O}(X)).$
An algebraic action of the one dimensional unipotent group $\mathcal{G}%
_{a}=(k,+)$ is conveniently described through the action of a locally
nilpotent derivation $D$ of $\mathcal{O}(X)$. Specifically, for $u\in 
\mathcal{G}_{a}=k$, we have the automorphism $u^{\ast }$ acting on $\mathcal{%
O}(X)$ and it is well-known (see for example \cite{essenboek} page 16-17)
that there exists a unique locally nilpotent derivation $D:\mathcal{O}%
(X)\longrightarrow \mathcal{O}(X)$ such that $u^{\ast }=\exp (uD)$. (One can
obtain $D$ by taking $D(f)=\frac{u^{\ast }f-f}{u}|_{u=0}.$) \ Similarly, if $%
\mathcal{G}_{a}^{n}$ acts on $X$, then we have for each component $\mathcal{G%
}_{a}$-action a locally nilpotent derivation $D_{i}$, and for each element $%
u=(u_{1},\ldots ,u_{n})\in \mathcal{G}_{a}^{n}$ we have the derivation $%
D:=u_{1}D_{1}+\ldots +u_{n}D_{n}$. If the action is faithful, there is a
canonical isomorphism of $\func{Lie}(\mathcal{G}_{a}^{n})$ with $%
kD_{1}+\ldots +kD_{n}$. In this case, the $D_{i}$ commute.

The situation is similar for a general unipotent group action $U\times
X\longrightarrow X$. \ Because the action is algebraic, each $f\in \mathcal{O%
}(X)$ is contained in a finite dimensional $U$ stable subspace $V_{f}$ on
which $U$ acts by linear transformations. \ Since $U$ is unipotent, for each 
$u\in U,$ $u^{\ast }-id$ is nilpotent on $V_{f}$ , so that $\ln
(u)(g)=\tsum\limits_{j=1}^{\infty }\frac{(u^{\ast }-id)^{j}g}{j}$ is a
finite sum for all $g\in $ $V_{f}.$ \ One checks that $D_{u}\equiv \ln (u)$
is a (locally nilpotent) derivation of $\mathcal{O}(X)$ and $u^{\ast }=\exp
(D_{u})$ . If the action is faithful, i.e. $U\rightarrow Aut(X)$ is
injective, there is a canonical isomorphism of $\func{Lie}(U)$ with $%
\{D_{u}~|~u\in U\}$. In fact, $\func{Lie}(U)=kD_{1}+\ldots +kD_{m}$ ($m=\dim
(U))$ for some locally nilpotent derivations $D_{i}$. In general the $D_{i}$
do not commute. In fact, all of them commute if and only if $U=\mathcal{G}%
_{a}^{m}$.

Two useful facts about unipotent group actions on quasiaffine varieties $V$
can be immediately derived from these observations:

\begin{enumerate}
\item Because each $u\in U$ acts via a locally nilpotent derivation of $%
O(V), $ the ring of invariants $O(V)^{U}$ is the intersection of the kernels
of locally nilpotent derivations.

\item Since kernels of locally nilpotent derivations are factorially closed,
their intersection is too, i.e. $O(V)^{U}$ is factorially closed. In
particular if $O(V)$ is a UFD\ so is $O(V)^{U}.$
\end{enumerate}

We will use the fact that $U$ is a special group in the sense of Serre. \
This means that a $U$ action which is locally trivial for the \'{e}tale
topology is locally trivial for the Zariski topology. \ If $G$ is a group
acting on a variety $X$, we denote by $X//G$ the algebraic quotient $X//G:=$%
\textbf{Spec }$\mathcal{O}(X)^{G}$ and by $X/G$ the geometric quotient (when
it exists). By a free action we mean an action for which the isotropy
subgroup of each element consists only of the identity. \ (A free action is
faithful.)

The paper is organized as follows: Section 2 contains some examples which
illustrate the main results and clarify their hypotheses. The main results
are proved in Section 3 from a geometric perspective, and Section 4 gives
them an algebraic interpretation. \ (The algebraic and geometric viewpoint
both have their merits: the geometric viewpoint lends itself to possible
generalizations, while the algebraic proofs are constructive and can be more
easily used in algorithms.) In section 5 we elaborate on some implications
of the main results for the Sathaye conjecture, and on the motivation for
studying this problem.

\section{Examples}

\label{Sec5}

The following examples are valuable in various parts of the text.

\begin{example}
\label{Ex1} Let $X=k^{3}$, and $U:=\{u_{a,b,c}~|~a,b,c\in k\}$ where 
\begin{equation*}
u_{a,b,c}:=%
\begin{pmatrix}
1 & a & b \\ 
0 & 1 & c \\ 
0 & 0 & 1%
\end{pmatrix}%
\end{equation*}%
a unipotent group acting by $u_{a,b,c}(x,y,z)=(x+a,y+az+b,z+c)$ (which
indeed is an algebraic action). For each $(a,b,c)\in k^{3}$ we thus have an
automorphism, and its associated derivation on $k[X,Y,Z]$ is $%
D_{a,b,c}=a\partial _{X}+(aZ+b-\frac{ac}{2})\partial _{Y}+c\partial _{Z}$.
Set $D_{1}=\partial _{Y},D_{2}:=\partial _{X}+Z\partial _{Y},D_{3}=\partial
_{Z}.$ \ As a Lie algebra $\func{Lie}(U)$ is generated by $D_{1,}D_{2},D_{3}$%
. \ One checks that $D_{1}$ commutes with $D_{2},D_{3}$, but $%
[D_{2},D_{3}]=D_{1}$. However, restricted to $k[X,Y,Z]^{D_{1}}=k[X,Z]$, $%
D_{2}$ and $D_{3}$ do commute, as they coincide with the derivations $%
\partial _{X}$ and $\partial _{Z}$. Furthermore, as a $k$ vector space $%
\func{Lie}(U)$ has basis $\partial _{X},\partial _{Y},\partial _{Z}$.
\end{example}

\begin{example}
\label{Ex2} Let $\mathcal{O}(X)=A=k[X,Y,Z]$, and $D_1=Z\partial_X,
D_2=\partial_Y$. These locally nilpotent derivations generate a $U=(\mathcal{%
G}_a)^2$-action on $k^3$ given by $(a,b)\cdot (x,y,z)\longrightarrow (x+az,
y+b, z)$. Now $k[Z]=A^{D_1,D_2}=\mathcal{O}(X/U)$. $D_1,D_2$ are linearly
independent over $k[Z]$. When calculating modulo $Z-\alpha$ where $\alpha\in
k$, we notice that $D_1\mod (Z-\alpha), D_2\mod (Z-\alpha)$ are linearly
independent over $A/(Z-\alpha)$ except when $\alpha=0$. However, defining $%
\mathcal{M}:=(\func{Lie}(U)\otimes k(Z))\cap {DER}(A)= (k(Z)D_1+k(Z)D_2)\cap 
{DER}(A)$ we see that $\mathcal{M}=k[Z]\partial_X+k[Z]\partial_Y$. The
derivations $\partial_X,\partial_Y$ are linearly independent modulo each $%
Z-\alpha$. And for each $\alpha\in k$, we have $A/(Z-\alpha)\cong k^{[2]}$.
\end{example}

\begin{example}
\label{Ex3} Let $P:=X^{2}Y+X+Z^{2}+T^{3}$, $\mathcal{X}:=%
\{(x,y,z,t)~|~P(x,y,z,t)=0\}$. Let $A:=k[x,y,z,t]:=k[X,Y,Z,T]/(P)=\mathcal{O}%
(\mathcal{X})$. The commuting locally nilpotent derivations $2Z\partial
_{Y}-X^{2}\partial _{Z}$, $3T^{2}\partial _{Y}-X^{2}\partial _{T}$ on $%
k[X,Y,Z,T]$ map $P$ to zero, and hence induce derivations $D_{1},D_{2}$ on $%
A $. They are linearly independent over $A^{D_{1},D_{2}}=k[X]$ and since
they commute, induce a $(\mathcal{G}_{a})^{2}$-action on $\mathcal{X}$.
Modulo $X-\alpha $, $D_{1},D_{2}$ are linearly independent, except when $%
\alpha =0$. Now defining $\mathcal{M}:=(\func{Lie}(U)\otimes k(X))\cap {DER}%
(A)=k[X]D_{1}+k[X]D_{2}=\func{Lie}(U)\otimes k[X]$, we see that $\mathcal{M}$
modulo $X-\alpha $ is a $k$-module of dimension $2$ except when $\alpha =0$,
when it is of dimension 1. Also, $A/(X-\alpha )\cong k^{[2]}$ except when $%
\alpha =0$, when it is isomorphic to $R[X]$ where $R=k[Z,T]/(Z^{2}+T^{3})$.
\end{example}

\begin{example}
The $U=\mathcal{G}_{a}\times \mathcal{G}_{a}$ action on $\mathbb{A}^{2}(k)$
given by 
\begin{equation*}
U\times \mathbb{A}^{2}\ni ((s,t),(x,y)\mapsto (x,y+t+sx)\in \mathbb{A}^{2}
\end{equation*}%
is faithful and fixed point free. \ However every point in $\ \mathbb{A}^{2}$
has a nontrivial isotropy subgroup. \ If $x\neq 0,$ then $%
((s,-sx),(x,y)\mapsto (x,y)$ and $((s,0),(0,y))\mapsto (0,y).$
\end{example}

\section{Main Results}

\label{Sec2}

The following simple lemma is useful in a number of places.

\begin{lemma}
Let $U$ be a unipotent algebraic group acting algebraically on a factorial
quasiaffine variety $X$ of dimension $n$ satisfying $\mathcal{O}(X)^{\ast
}=k^{\ast .}$ \ If the action is not transitive and some point $x\in X$ has
orbit of dimension $n-1,$ then $\mathcal{O}(X)^{U}=k[f]$ for some $f\in 
\mathcal{O}(X)$
\end{lemma}

\begin{proof}
Since $n-1$ is the maximum orbit dimension there is a Zariski open subset $V$
of $X$ for which the geometric quotient $V/U$ exists as a variety. \ Then
the transcendence degree of the quotient field $K$ of $\mathcal{O}(V/U)\ $is
equal to 1. Since $K=qf(\mathcal{O}(X)^{U})$ and 
\begin{equation*}
\mathcal{O}(X)^{U}=\mathcal{O}(X)\cap K,
\end{equation*}%
is a ufd, $\mathcal{O}(X)^{U}$ is finitely generated over $k.$ From $(%
\mathcal{O}(X)^{U})^{\ast }=k^{\ast }$, we conclude that $\mathcal{O}%
(X)^{U}=k[f]$ for some $f\in \mathcal{O}(X)$
\end{proof}

\subsection{Unipotent actions having zero-dimensional quotient}

\begin{theorem}
\label{FaithUni} Let $U$ be an $n$-dimensional unipotent group acting
faithfully on an affine $n$-dimensional variety $X$ satisfying $\mathcal{O}%
(X)^{\ast }=k^{\ast }$. If either\linebreak\ 
\begin{tabular}{ll}
a) & Some $x\in X$ has trivial isotropy subgroup or \\ 
b) & $n=2$, $X$ is factorial, and $U$ acts without fixed points,%
\end{tabular}%
\linebreak\ then the action is transitive. $\ $In particular $X\cong k^{n}$.
\end{theorem}

\begin{proof}
In case a) there is an open affine subset $V$ of $X$ on which $U$ acts
without fixed points. \ Since $U$ has the same dimension as $V$, $V//U$ is
zero-dimensional, hence $\mathcal{O}(V//U)$ is a field. This field contains $%
k$, and its units are contained in $\mathcal{O}(X)^{\ast }=k^{\ast }$, hence 
$\mathcal{O}(V//U)=k$. It follows that there exists an open set $V^{\prime }$%
of $X$ for which $V^{\prime }/U\cong $\textbf{Spec}$\ k$. Thus $V^{\prime
}\cong U$ as a variety, and therefore $V^{\prime }\cong k^{n}.$ If $v\in
V^{\prime }$, then $Uv=V^{\prime }$. Since $U$ is unipotent, all orbits are
closed, hence $V^{\prime }$ is closed in $X$. Since it is of dimension $n$,
and $X$ is irreducible of dimension $n$, we have that $V^{\prime }=X$.

In case b) $X$ is necessarily smooth since it is smooth in codimension 1 and
every orbit is infinite. \ If $X$ has a two dimensional (i.e. dense) orbit
then the conclusion follows as in a). \ So we assume\ for each $x\in X$ that
the orbit $Ux$ is one dimensional, given as $\exp (uD)x,$ and therefore
isomorphic to $\mathbb{A}^{1}(k)$ by the discussion in the introduction.
From Lemma 1 we conclude that $\mathcal{O}(X)^{U}=k[f]$ for some $f\in 
\mathcal{O}(X)$

Note that factorial closure of $\mathcal{O}(X)^{U}$ implies that $f-\lambda $
is irreducible for every $\lambda \in k.$ The absence of nonconstant units
implies that $X\rightarrow $\textbf{Spec}$(k[f])$ is surjective, and all
fibers are $U$ orbits. Smoothness of $X$ implies in addition that this
mapping is flat, hence an $\mathbb{A}^{1}$ bundle over $\mathbb{A}^{1}.$ \
But any such bundle is trivial, so we conclude again that $X\cong \mathbb{A}%
^{2}.$ \ 
\end{proof}

Example 4 of the previous section illustrates case \emph{b)}.

\subsection{Unipotent actions having one-dimensional quotient}

The following theorem is the main result of this paper.

\begin{theorem}[\textbf{Main theorem}]
\label{General} Let $U$ be a unipotent algebraic group of dimension $n$,
acting on $X$, a factorial variety of dimension $n+1$ satisfying $\mathcal{O}%
(X)^{\ast }=k^{\ast }$.\newline
(1) If at least one $x\in X$ has trivial stabilizer then $\mathcal{O}(X)^{U}=%
\mathcal{O}(X/\!\!/U)=k[f]$. Furthermore, $f^{-1}(\lambda )\cong k^{n}$ for
all but finitely many $\lambda \in k$. \newline
(2) If $U$ acts freely, then $X$ is $U$-isomorphic to $U\times k$. In
particular, $X\simeq k^{n+1}$ and $f$ is a coordinate.
\end{theorem}

An important example to keep in mind is example \ref{Ex1}, as this satisfies
(1) but not (2). (There $U=\mathcal{G}_{a}^{2}$.)

\begin{proof}[Proof of theorem \protect\ref{General}]
\ \newline
\textsc{claim 1:} $\mathbf{\ }\mathcal{O}(X)^{U}=k[f]$. \newline
\emph{Proof of claim 1:} This follows from lemma 1.

\textsc{claim 2:} $f:X\rightarrow k$ is surjective and has fibers isomorphic
to $U$. The fibers are the $U$-orbits.\newline
\emph{Proof of claim 2:} The fibers $f^{-1}(\lambda )$ are the zero loci of
the irreducible $f-\lambda $, and are invariant under $U$. Since $U$ acts
freely on each fiber and orbits of unipotent group actions are closed, we
see that the $f$ fibers are exactly the $U$ orbits in $X$. Thus $f$ is a $U$%
-fibration (and, as the underlying variety of $U$ is $k^{n}$, an $\mathbb{A}%
^{n}$-fibration).

\textsc{claim 3:} $X$ is smooth.\newline
\emph{Proof of claim 3:} The set $X_{sing}$ is $U$-stable, hence it is a
union of $U$-orbits. The $U$-orbits are the zero sets $f-\lambda $, hence of
codimension 1. So $X_{sing}$ is of codimension 1 or empty. But $X$ is
factorial, so in particular normal, which implies that the set of singular
points of $X$, denoted by $X_{sing}$, is of codimension at least 2. This
means that $X_{sing}$ can only be empty.

\textsc{claim 4:} $f$ is smooth.\newline
\emph{Proof of claim 4:} All fibers of $f$ are isomorphic to $U$, hence to $%
k^{n}$, by claim 2. Thus the fibers of $f$ are geometrically regular of
dimension $n$. $\mathbf{\ \ }$Since $X$ is smooth, $f$ is flat, and
proposition 10.2 of \cite{Hartshorne} yields that $f$ is smooth.

\textsc{claim 5:} $X\times_f X$ is smooth.\newline
\emph{Proof of claim 5:} $X\times_f X$ is smooth since it is a base
extension of the smooth $X$ by the smooth morphism $f$.

\textsc{claim 6:} $g:U\times X\rightarrow X\times _{f}X$ given by $%
(u,x)\mapsto (x,ux)$ is an isomorphism.\newline
\emph{Proof of claim 6:} The map $g$ restricted to $U\times f^{-1}(\lambda )$
is a bijection onto $\{(x,y)~|~f(x)=f(y)=\lambda \}$. Taking the union over $%
\lambda \in k$, we get that $g$ is a bijection. Since both $U\times X$ and $%
X\times _{f}X$ are smooth and $g$ is a bijection, Zariski's Main Theorem
implies that $g$ is an open immersion if it is birational. \ If so then $g$
must be an isomorphism since it is bijective.

From Rosenlicht's cross section theorem \cite{Rosen}, $X$ has a $U$ stable
open subset $\tilde{X}$ on which the $U$ action has a geometric quotient $%
\tilde{X}/U$ and a $U$ equivariant isomorphism $\tilde{X}\cong U\times 
\tilde{X}/U$. \ Restricting $g$ to $U\times \tilde{X}\rightarrow \tilde{X}%
\times _{f}\tilde{X}$ \ is clearly an isomorphism, so that $g$ is birational.

Now we are ready to prove the theorem. Using def. 0.10 p.16 of \cite{GIT},
and the fact (4) that $f$ is smooth, together with (6), yields that $%
f:X\rightarrow \mathbb{A}^{1}$ is an \'{e}tale principal $U$-bundle and
therefore a Zariski locally trivial principal $U$ bundle as $U$ is special.
Such bundles are classified by the cohomology group $H^{1}(U,k)$, which is
trivial because $U$ is unipotent. $\ $Thus the bundle $f:X\rightarrow k$ is
trivial, which means that $X\cong U\times k.$
\end{proof}

\begin{remark}
\begin{enumerate}
\item To obtain $\tilde{X}$ explicitly and avoid the use of Rosenlicht's
theorem, recall that the action of $U$ is generated by a finite set of $%
\mathbb{\ }\mathcal{G}_{a}$ actions each one given as the exponential of
some locally nilpotent derivation $D_{i}$ of $\mathcal{O}(X),$ indeed $%
D_{i}\in \mathfrak{u}$, the Lie algebra of $U$. As such there is an open
subset $X_{i}$ of $X$ on which $D_{i}$ has a slice, and the corresponding $%
\mathcal{G}_{a}$ acts by translation. Then $\tilde{X}:=\cap _{i=1}^{s}X_{i}$.

\item One can avoid the use of the \'{e}tale topology by applying a
"Seshadri cover" \cite{Sesh}. \ One constructs a variety $Z$ finite over $X,$
necessarily affine, to which the $U$ action extends so that

\begin{enumerate}
\item $k(Z)/k(X)$ is Galois. Denote the Galois group by $\ \Gamma .$

\item The $\Gamma $ and $U$ actions commute on $Z.$

\item The $U$ action on $Z$ is Zariski locally trivial and, because the
action on $X$ is proper by claim 6,

\item $Y\equiv Z/U$ exists as a separated scheme of dimension 1, hence a
curve, and affine because of the existence of nonconstant globally defined
regular functions, namely $\mathcal{O}(Z)^{U}.$

\item $\mathcal{O}(X)^{U}\cong \mathcal{O}(Y)^{\Gamma }$and $X//U\cong
X/U\cong Y/\Gamma $ shows that $X\rightarrow X/U$ is Zariski locally trivial.
\end{enumerate}
\end{enumerate}
\end{remark}

\section{Algebraic Version}

\label{Sec3}

\subsection{Unipotent actions having zero dimensional kernel}

Let $X$ be a quasiaffine variety, and $U$ an algebraic group acting on $X$.
We write $A:=\mathcal{O}(X)$ and denote by $\mathfrak{u}$ the Lie algebra of 
$U.$ In this section, we will make the following assumptions:

\begin{tabular}{ll}
(P)~~ & \emph{a) }$X$ and $U$ are of dimension $n$. \\ 
& \emph{b) }There is a point $x\in X$ such that $\func{stab}(x)=\{e\}$. \\ 
& \emph{c)} $\mathcal{O}(X)^{\ast }=k^{\ast }$%
\end{tabular}

\begin{definition}
Assume (P). We say that $D_{1},\ldots ,D_{n}$ is a triangular basis of $%
\mathfrak{u}$ (with respect to the action on $X$) if
\end{definition}

\begin{enumerate}
\item $\mathbf{\ }\mathcal{\ }\mathfrak{u}=kD_{1}\oplus kD_{2}\oplus \ldots
\oplus kD_{n}$ and

\item With subalgebras $A_{i}$ of $A$ given by $A_{1}:=A$, $%
A_{i}:=A^{D_{1}}\cap \ldots \cap A^{D_{i-1}}$, the restriction of $D_{i}$ to 
$A_{i}$ commutes with the restrictions of $D_{i+1},\ldots ,D_{n}$.
\end{enumerate}

For a triangular basis, it is clear that $D_{j}(A_{j})\subseteq A_{j}$ for
each $j$.

If $U$ is unipotent then the existence of a triangular basis is a
consequence of the Lie-Kolchin theorem. Indeed, the Lie algebra $\mathfrak{%
u\ }$of $U$ is isomorphic to a Lie subalgebra of the full Lie algebra of
upper triangular matrices over $k.$ In particular $\mathfrak{u}$ has a basis 
$D_{1},\ldots ,D_{n}$ satisfying $[D_{i},D_{j}]\in span\{D_{1},\ldots
D_{\min \{i,j\}-1}\}..$ By definition of the $A_{i}$ this basis is
triangular with respect to the action and $D_{1}$ is in the center of $%
\mathfrak{u.}$

\begin{proposition}
\label{tribas} Assume (P) and $U$ unipotent. Then $A\cong k[s_{1},\ldots
,s_{n}]=k^{[n]}$ where $D_{i}(s_{i})=1$, and $D_{i}(s_{j})=0$ if $j>i$.
\end{proposition}

\begin{proof}
We proceed by induction $n=\dim $ $\mathfrak{u}$. If $n=1$, then we have one
nonzero LND on a dimension one $k$-algebra domain $A$ satisfying $A^{\ast
}=k^{\ast }$. It is well-known that this means that $A\cong k[x]$ and the
derivation is simply $\partial _{x}$. Suppose the theorem is proved for $n-1$%
. Let $D_{1},D_{2,}\ldots D_{n}$ be a triangular basis for $\mathfrak{u}$.
Restricting to $A^{D_{1}}$ and noting that $D_{1}$ is in the center of $%
\mathfrak{u},$ we have an action of the Lie algebra $\mathfrak{u}/kD_{1}$
which has the triangular basis $k\overline{D_{2}}+\ldots +k\overline{D_{n}}$
\ ($\overline{D_{i}}$ denotes residue class modulo $kD_{1}).$ By
construction $\overline{D}_{i}(a):=D_{i}(a)$ is well defined, and by
induction we find $s_{2},\ldots ,s_{n}\in $ $A^{D_{1}}$ satisfying $%
D_{i}(s_{i})=1,D_{i}(s_{j})=0$ if $j>i\geq 2$. $D_{i}(s_{j})=\delta _{ij}$%
.\smallskip \smallskip\ \newline
Next we consider a preslice $p\in A$ such that $D_{1}(p)=q,D_{1}(q)=0$, i.e. 
$q=q(s_{2},\ldots ,s_{n})$. We pick $p$ in such a way that $q$ is of lowest
possible lexicographic degree w.r.t $s_{2}>>s_{3}>>\ldots >>s_{n}.$ Now $%
D_{1}(D_{2}(p))=D_{2}D_{1}(p)=D_{2}(q)$. Restricted to $k[s_{2},\ldots
,s_{n}]$, $D_{2}=\partial _{s_{2}}$, so $D_{2}(q)$ is of lower $s_{2}$%
-degree than $q$. Unless $D_{2}(q)=0$, we get a contradiction with the
degree requirements of $q$, as $D_{2}(p)$ would be a \textquotedblleft
better\textquotedblright\ preslice having a lower degree derivative. Thus, $%
q\in k[s_{3},\ldots ,s_{n}]$. Using the same argument for $D_{3},D_{4}$ etc.
we get that $q\in k^{\ast }$. Hence, $p$ is in fact a slice.
\end{proof}

\subsection{Unipotent actions having one-dimensional quotient}

With the same notations as in the previous section, we also denote the ring
of $U$ invariants in $A$ by $A^{U}$ and \textbf{Spec }$A^{U}$ \ by $X//U$.
Note that $A^{U}=\{a\in A~|~D(a)=0$ for all $D\in \mathfrak{u}\}$. If $U$ is
unipotent and $D_{1},\ldots ,D_{n}$ is a triangular basis of $\mathfrak{u}$,
we again write $A_{1}:=A$, $A_{i+1}=A_{i}\cap A^{D_{i}}$, noting that $%
A^{U}=A_{n}.$ In this section we consider the conditions :

\begin{tabular}{ll}
(Q1)~~ & $U$ is a unipotent algebraic group of dimension $n$ acting on an \\ 
& affine variety $X$ of dimension $n+1$ with $A^{\ast }=k^{\ast }$.%
\end{tabular}

and:

\begin{tabular}{ll}
(Q) & $A^{U}=k[f]$ for some irreducible $f\in A\backslash k$.%
\end{tabular}

\begin{remark}
According to Lemma 1, condition (Q1) along with the assumption that $X$ is
factorial and the existence of a point $x\in X$ with $\func{stab}(x)=\{e\}$,
implies that (Q) holds.
\end{remark}

\begin{notation}
Assuming (Q), let $\alpha \in k$. Set $\overline{A}:=A/(f-\alpha )$ and
write $\overline{a}$ for the residue class of $a$ in $\overline{A}$ and $%
\overline{D}$ for the derivation induced by $D\in \mathfrak{u}$ on $%
\overline{A}.$
\end{notation}

Our goal is to prove the following constructively:

\begin{theorem}
\label{CDH2a}\label{CDH2} Assume (Q1) and (Q). Let $D_{1},\ldots ,D_{n}$ be
a triangular basis of $\mathfrak{u}$.

\begin{enumerate}
\item For $\alpha \in k,$

\begin{enumerate}
\item If $\ \overline{D_{1}},\ldots ,\overline{D_{n}}$ are independent over $%
A/(f-\alpha )$, then 
\begin{equation*}
A/(f-\alpha )\cong k^{[n]}.
\end{equation*}

\item There are only finitely many $\alpha $ for which $\overline{D_{1}}%
,\ldots ,\overline{D_{n}}$ are dependent over $A/(f-\alpha )$.
\end{enumerate}

\item In the case that $\overline{D_{1}},\ldots ,\overline{D_{n}}$ are
independent over $A/(f-\alpha )$ for each $\alpha \in k$, then there are $%
s_{1},\ldots ,s_{n}\in A$ with $A=k[s_{1},\ldots ,s_{n},f]$, hence $A$ is
isomorphic to a polynomial ring in $n+1$ variables (and $f$ is a coordinate).
\end{enumerate}
\end{theorem}

\begin{definition}
Assume (Q1) and (Q), and a triangular basis $D_{1},\ldots ,D_{n}$ of $%
\mathfrak{u}$. Define 
\begin{equation*}
\mathcal{P}_{i}:=\{p\in A~|~D_{i}(p)\in k[f],\text{ }D_{j}(p)=0\text{ if }%
j<i\}
\end{equation*}%
and%
\begin{equation*}
\mathcal{J}_{i}:=D_{i}(\mathcal{P}_{i})\subseteq k[f].
\end{equation*}
\end{definition}

Thus $\mathcal{P}_{i}$ is the set of "preslices" of $D_{i}$ that are
compatible with the triangular basis $D_{1},\ldots ,D_{n}.$

\begin{lemma}
\label{AA1} There exist $p_{i}\in \mathcal{P}_{i}\backslash \{0\},p_{i}\in
A_{i},$ and $q_{i}\in k^{[1]}\backslash \{0\}$ such that $\mathcal{J}%
_{i}=q_{i}(f)k[f]$ and $D_{i}(p_{i})=q_{i}$.
\end{lemma}

\begin{proof}
First note that $\mathcal{J}_{i}$ is not empty, as theorem \ref{tribas}
applied to $A(f):=A\otimes k(f)$ gives an $s_{i}\in A(f)$ which satisfies $%
D_{i}(s_{i})=1$, $D_{j}(s_{i})=0$ if $j<i$. Multiplying $s_{i}$ by a
suitable element of $k[f]$ gives a nonzero element $r(f)s_{i}$ of $\mathcal{P%
}_{i}$, and $D_{i}(r(f)s_{i})=r(f)$. Because $k[f]=\cap \ker (D_{i})$, $%
\mathcal{P}_{i}$ is a $k[f]$-module, and therefore $\mathcal{J}_{i}$ is an
ideal of $k[f]$. This means that $\mathcal{J}_{i}$ is a principal ideal, and
we take for $q_{i}$ a generator (and $p_{i}\in D_{i}^{-1}(q_{i})$). Since $%
D_{j}(p_{i})=0$ if $j<i$, we have $p_{i}\in A_{i}$.
\end{proof}

\begin{corollary}
The $p_{i},$ $1\leq i\leq n,$ are algebraically independent over $k.$
\end{corollary}

\begin{proof}
The $s_{i}$ are certainly algebraically independent, and $p_{i}\in
k[f]s_{i}. $
\end{proof}

\begin{lemma}
\label{BLA} Assume (Q), and take $p_{i},q_{i}$ as in lemma \ref{AA1}. Then
the $D_{i}$ are linearly dependent modulo $f-\alpha $ if and only if $%
q_{i}(\alpha )=0$ for some $i$.
\end{lemma}

\begin{proof}
($\Rightarrow $): \ Suppose that $0\not=D:=g_{1}D_{1}+\ldots +g_{n}D_{n}$
satisfies $\overline{D}=0$ where $g_{i}\in A$, and not all $\overline{g}_{i}=%
\bar{0}$. Let $i$ be the highest such that $\bar{g}_{i}\not=\bar{0}$. Then $%
0=\overline{D(p_{i})}=\bar{g}_{i}\bar{D}\bar{p}_{i}=\bar{g}_{i}\overline{%
q_{i}(f)}$. Since $\bar{A}$ is a domain, $q_{i}(\alpha )=\overline{q_{i}(f)}%
=0$. \newline
($\Leftarrow $): Assume $f-\alpha $ divides $q_{i}(f)$. We need to show that
the $\overline{D_{i}}$ are linearly dependent over $A/(f-\alpha )$. Consider 
$\bar{D}_{i}$ restricted to $\bar{A}_{i}$. If $j>i$ then $\bar{D}_{i}(p_{j})=%
\overline{D_{i}(p_{j})}=\overline{0}$. Furthermore $\bar{D}_{i}(\bar{p}_{i})=%
\bar{q}_{i}(f)=q(\alpha )=0$. Hence, $\bar{D}_{i}$ is zero if restricted to $%
k[\bar{p}_{i},\ldots ,\bar{p}_{n}]$. But since this is of transcendence
degree $n$, it follows that $\bar{D}_{i}=0$ on $\bar{A}_{i}$. \ Reversing
the argument of ($\Rightarrow $) yields the linear dependence of the $\bar{D}%
_{i}.$
\end{proof}

\begin{proof}
\emph{(of theorem \ref{CDH2a})} Part 1: If $\bar{D}_{1},\ldots ,\bar{D}_{n}$
are independent, then Proposition 1 yields that $\bar{A}\cong k^{[n]}$.
Lemma \ref{BLA} states that for any point $\alpha $ outside the zero set of $%
q_{1}q_{2}\cdots q_{n}$ we have $A/(f-\alpha )\cong k^{[n]}$. This zero set
is either all of $k$ or finite, yielding part 1.\newline
Part 2: Lemma \ref{BLA} tells us directly that for each $1\leq i\leq n$ and $%
\alpha \in k$, we have $q_{i}(\alpha )\not=0$. But this means that the $%
q_{i}\in k^{\ast }$, so the $p_{i}$ can be taken to be actual slices ($%
s_{i}=p_{i})$. \ Using the fact that $s_{i}\in A_{i}$ we obtain that $%
A=A_{1}=A_{2}[s_{1}]=A_{3}[s_{2},s_{1}]=\ldots =A_{n+1}[s_{1},\ldots
s_{n}]=k[s_{1},\ldots ,s_{n},f]$ as claimed$.$
\end{proof}

\section{Consequences of the main theorems}

\label{Sec4}

This paper is originally motivated by the following result of \cite{Mau}:

\begin{theorem}
Let $A=k[x,y,z]$ and $D_{1},D_{2}$ two commuting locally nilpotent
derivations on $A$ which are linearly independent over $A$. Then $%
A^{D_{1},D_{2}}=k[f]$ and $f$ is a coordinate.
\end{theorem}

Here the notation $A^{D_{1},D_{2}}$ means $A^{D_{1}}\cap A^{D_{2}}$ the
intersections of the kernels of $D_{1}\ $and $D_{2}$, which is the set of
elements vanishing under $D_{1}$ resp. $D_{2}$. (Note that for the $\mathcal{%
G}_{a}$ action associated to $D$, this notation means $\mathcal{O}(X/%
\mathcal{G}_{a})=\mathcal{O}(X)^{\mathcal{G}_{a}}=\mathcal{O}(X)^{D}$). By a 
\textbf{coordinate} is meant an element $f$ in $k^{[n]}$ for which there
exist $f_{2},\ldots ,f_{n}$ with $k[f,f_{2},\ldots ,f_{n}]=k^{[n]}$.
Equivalently, $(f,f_{2},\ldots ,f_{n}):k^{[n]}\longrightarrow k^{[n]}$ is an
automorphism. The most important ingredient of this theorem is Kaliman's
theorem \cite{Kal}.

In \cite{Mau} it is conjectured that this result is true also in higher
dimensions, i.e. having $n$ commuting linearly independent locally nilpotent
derivations on $k^{[n+1]}$ should yield that their common kernel is
generated by a coordinate. However, it seems that this conjecture is very
hard, on a par with the well-known Sathaye conjecture:

\textbf{$SC(n)$ Sathaye-conjecture:} Let $f\in A:=k^{[n]}$ such that $%
A/(f-\lambda )\cong k^{[n-1]}$. Then $f$ is a coordinate.\newline

The Sathaye conjecture is proved for $n\leq 3$ by the aforementioned
Kaliman's theorem. \ Therefore, the original motivation was to find
additional requirements in higher dimensions to achieve the result that $f$
is a coordinate. The results in this paper give one such requirement, namely
that $k^{[n]}/(f-\lambda )\cong k^{[n-1]}$ for all constants $\lambda $.

Another consequence of the result of this paper is that the Sathaye
conjecture is equivalent to \newline

\noindent \noindent \textbf{$MSC(n)$ Modified Sathaye Conjecture:} Let $%
A:=k^{[n]}$, and let $f\in A$ be such that $A/(f-\alpha )\cong k^{[n-1]}$
for all $\alpha \in k$. Then there exist $n-1$ commuting locally nilpotent
derivations $D_{1},\ldots ,D_{n-1}$ on $A$ such that $A^{D_{1},\ldots
,D_{n-1}}=k[f]$ and the $D_{i}$ are linearly independent modulo $(f-\alpha )$
for each $\alpha \in k$.

\begin{proof}[Proof of equivalence of $SC(n)$ and $MSC(n)$.]
Suppose we have proven the $MSC(n)$. Then for any $f$ satisfying
\textquotedblleft $A/(f-\alpha )\cong k^{[n-1]}$ for all $\alpha \in k$%
\textquotedblright\ we can find commuting LNDs $D_{1},\ldots ,D_{n-1}$ on $A 
$ giving rise to a $\mathcal{G}_{a}^{n-1}$action satisfying the hypotheses
of Theorem 2. Applying this theorem, we obtain that $f$ is a coordinate in $%
A $. So the $SC(n)$ is true in that case.

Now suppose we have proven the $SC(n)$. Let $f$ satisfy the requirements of
the $MSC(n)$ , that is, \textquotedblleft $A/(f-\alpha )\cong k^{[n-1]}$ for
all $\alpha \in k$\textquotedblright . Since $f$ satisfies the requirements
of $SC(n)$, $f$ then must be a coordinate. So it has $n-1$ so-called mates: $%
k[f,f_{2},\ldots ,f_{n}]=k^{[n]}$. But then the partial derivative with
respect to each of these $n$ polynomials $f,f_{2},\ldots ,f_{n}$ defines a
locally nilpotent derivation. \ All of them commute, and the intersection of
the kernels of the last $n-1$ derivations is $k[f]$; so the MSC holds.
\end{proof}


\begin{thebibliography}{9}
\bibitem{essenboek} A. van den Essen: \ "Polynomial Automorphisms and the
Jacobian Conjecture". \ Birkhauser, Boston, 2000.

\bibitem{Hartshorne} R. Hartshorne: \ "Algebraic Geometry" Springer Verlag,
Berlin, 1977.

\bibitem{Kal} S. Kaliman: \ \emph{Polynomials with general }$\mathbb{C}^{2}$%
\emph{\ fibers are variables.} \ Pacific J. Math. \ 203 (2002) 161-189.

\bibitem{Mau} S. Maubach: \ \emph{The commuting derivations conjecture.} \
J. Pure Appl. Algebra 179 (2003) 159-168.

\bibitem{GIT} D. Mumford: "Geometric Invariant Theory" Academic Press, New
York, 1965.

\bibitem{Rosen} M. Rosenlicht: \ \emph{Another proof of a theorem on
rational cross sections. }Pacific J. Math. 20 (1967) 129-133.

\bibitem{Sesh} C.S. Seshadri: \ \emph{Quotient spaces modulo reductive
groups.} \ Ann. Math. 95 (1972) 511-556.
\end{thebibliography}
\end{document}